\documentclass[12pt,leqno]{amsart}%
\usepackage{amsmath}
\usepackage[margin=1.5in]{geometry}
\usepackage{graphicx}
\usepackage{amsfonts}
\usepackage{amssymb}
\usepackage[utf8]{inputenc}
\usepackage{comment}
\usepackage{hyperref}%
\providecommand{\U}[1]{\protect\rule{.1in}{.1in}}
\newtheorem{theorem}{Theorem}[section]
\theoremstyle{plain}

\newtheorem{corollary}{Corollary}[section]

\newtheorem{lemma}{Lemma}[section]

\newtheorem{proposition}{Proposition}[section]

\numberwithin{equation}{section}
\theoremstyle{definition}
\newtheorem{definition}{Definition}
\theoremstyle{remark}
\newtheorem{remark}{Remark}[section]

\title[Critical points of convex functionals]{ Regularity for critical points of convex functionals on Hessian spaces}
\author{Arunima Bhattacharya }
\address{Department of Mathematics\\
University of Washington, Seattle, WA 98195, U.S.A.}
\email{arunimab@uw.edu}

\begin{document}

\maketitle

\begin{abstract}
We consider variational integrals of the form $\int F(D^2u)$ where $F$ is convex and smooth on the Hessian space. We show that a critical point $u\in W^{2,\infty}$ of such a functional under compactly supported variations is smooth if the Hessian of $u$ has a small oscillation.
\end{abstract}

\section{ Introduction}

In this paper, we prove full regularity for critical points of variational integrals of the form \begin{equation}
\int_{B_1}F(D^{2}u)dx\label{Ffunc}%
\end{equation} where $B_1\subset\mathbb R^n$, by developing regularity theory for weak solutions of fourth order nonlinear elliptic equations in double divergence form.

 Variational problems for the volume functional naturally give rise to fourth order nonlinear equations in double divergence form and provide a significant geometric motivation behind the study of their regularity. For example, the Hamiltonian stationary equation:
\begin{align}
	\int_{B_1}\sqrt{\det g}g^{ij}\delta^{kl}u_{ik}\eta_{jl}dx=0
	\text{     }\forall\eta   \in C_{0}^{\infty}(B_{1}) \label{hstat}
	\end{align}
	where $g=I+\left(  D^{2}u\right)  ^{2}$ is the induced metric from the Euclidean metric in $\mathbb R^{2n}$,
 governs Lagrangian surfaces that minimize the area functional \begin{equation*}
\int_{B_1}\sqrt{\det(I+\left(  D^{2}u\right)  ^{2})}dx \label{HS}%
\end{equation*}
among potential functions $u$ (\cite{Oh}, \cite[Proposition 2.2]{SW03}). For variational integrals of the form (\ref{Ffunc}), the critical point under compactly supported variations satisfies an Euler-Lagrange equation, which shares a similar fourth order structure. In \cite{BW1}, we showed that if $F$ is a smooth convex function of the Hessian and can be expressed as a function of the square of the Hessian, then a $C^{2,\alpha}$ critical point (under compactly supported variations) of (\ref{Ffunc}) will be smooth. We achieved this by establishing regularity for a class of fourth order equations in the following double divergence form\begin{equation}
\int_{B_1}a^{ij,kl}(D^{2}u)u_{ij}\eta_{kl}dx=0,\text{ }\forall\eta\in
C_{0}^{\infty}(B_1) \label{eq1}.
\end{equation}
Recently in \cite{BCW}, we studied regularity for a certain class of fourth order equations in double divergence form that in turn lead to proving smoothness for any $C^{1}$-regular Hamiltonian stationary
Lagrangian submanifold in a symplectic manifold. There are well-known equations that share the same structure: bi-harmonic functions, extremal K\"ahler metrics, to name a few. For second order, PDE theory for equations in divergence form plays a significant role in geometric analysis and is well developed by now. In comparison, for fourth order, theory of double divergence form equations is largely unexplored but remains an important developing area of geometric analysis.

Our results in this paper show that a critical point $u\in W^{2,\infty}(B_1)$ of the variational integral ($\ref{Ffunc}$) will be smooth if $F$ is a uniformly convex, smooth function of the Hessian and the Hessian has a small oscillation.

Before presenting our main result, we first introduce the following notations and definition. 

 \noindent
 \textbf{Notations.} 
 Through out this paper $U$ denotes a convex neighborhood in $S^{n\times n}$, $V$ denotes a convex neighborhood in $\mathbb R^{n}$, and $B_r$ denotes a ball of radius $r$ centered at the origin in $\mathbb R^n$ unless specified otherwise.

\begin{definition}[Small oscillation]\label{bm}
    We say that $f\in L^{\infty}(B_1)$ has a small oscillation in $B_1$ with modulus $\omega_f$ if there exists a $\omega_f>0$ small such that 
    \begin{align}
       ||f-(f)_1||_{L^{\infty}(B_1)}\leq \omega_f \label{BMO1}
        \end{align}
         where $(f)_1$ denotes the average of $f$ over the ball $B_1$.
        \end{definition}

Our main result is the following.

\begin{theorem}\label{main1}[Fourth order]
Suppose that $u\in W^{2,\infty}(B_{1})$ is a critical point of
(\ref{Ffunc}) where
	$F$ is smooth and uniformly convex on $U$ and $D^2u(B_1)\subset U.$ If $D^2u$ has a small oscillation in $B_1$, then $u$ is smooth in $B_1$.
\end{theorem}

 Our proof of Theorem \ref{main1} goes as follows: we start by deriving the Euler Lagrange equation from compactly supported variations. We find the critical point $u$ to be a weak solution of a fourth order nonlinear equation in double divergence form. By a weak solution we mean each of the partial derivatives are taken in a
distributional sense. We develop regularity theory for such a class of fourth order equations by proving $u\in C^{2,\alpha}$, which is sufficient to achieve smoothness.  H\"older continuity of the Hessian leads to H\"older continuous coefficients, which in turn lead to a self improving solution. Finally using the convexity property of $F$, we apply the regularity theory developed to our critical point $u$ to achieve smoothness. 

We elucidate the technical difficulties of the above process by illustrating an analogous result for critical points of simpler variational integrals of the form 
\begin{equation}
    \int_{B_1}F(Du)=0, \label{2nd}
\end{equation}
without relying on the well-known De Giorgi-Nash theory.
\begin{theorem}\label{main2}[Second order]
Suppose that $u\in W^{1,\infty}(B_{1})$ is a critical point of
(\ref{2nd}) where
	$F$ is smooth and uniformly convex on $V$ and $Du(B_1)\subset V.$ If $Du$ has small a oscillation in $B_1$, then $u$ is smooth in $B_1$.
\end{theorem}

We prove the above result by studying second order nonlinear equations in divergence form \begin{equation}
\int_{B_1}a^{ij}(Du)u_{i}\eta_{j}dx=0,\text{ }\forall\eta\in
C_{0}^{\infty}(B_1) \label{eq2}.
\end{equation}
A well known example of such an equation is the minimal surface equation:
\begin{equation}\int_{B_1}\frac{\delta_{ij}}{\sqrt{1+|Du|^2}}u_i\eta_jdx=0\label{ms}
\end{equation}
which governs critical points of the area functional
\[\int_{B_1} \sqrt{1+|Du|^2}
\] on graphs $(x,u(x)).$ Following our strategy to prove $C^{2,\alpha}$ regularity for the fourth order equation (\ref{eq1}), we derive $C^{1,\alpha}$ estimates for the second order equation (\ref{eq2}), which again is sufficient to prove smoothness.

The organization of the paper is as follows: in section 2, we develop regularity theory for weak solutions of (\ref{eq1}) by first proving a $C^{2,\alpha}$ estimate for the solution. In section 3, we establish the second order analogue of the same. Finally, in section 4, we prove our main results.

\section{Regularity theory: fourth order}

In this section, we develop regularity theory for weak solutions of (\ref{eq1}). 
Denoting $h_{m}=he_m$, we start by introducing the following definition.
\begin{definition}[Regular equation]
	We define equation (\ref{eq1}) to be regular on $U$ when the 
	following conditions are satisfied on $U$:
	\begin{enumerate}
	    \item[(i)] The coefficients $a^{ij,kl}$ depend smoothly on $D^{2}u$.
	    \item[(ii)] The linearization of (\ref{eq1}) is uniformly elliptic: the leading coefficient of the linearized equation given by 
	    \begin{equation}
	        b^{ij,kl}(D^2u(x))=\int_{0}^1\frac{\partial}{\partial{u_{ij}}} \bigg[a^{pq,kl}(D^2u(x)+t [D^2u(x+h_m)-D^2u(x)])u_{pq}(x)\bigg]dt \label{Bdef}
	    \end{equation}
	     satisfies the standard Legendre ellipticity condition for any $\xi\in U$:
	    \begin{equation}
b^{ij,kl}(\xi)\sigma_{ij}\sigma_{kl}\geq\Lambda\left\Vert \sigma\right\Vert \text{  }\forall \sigma \in S^{n\times n}. 
\label{Bcondition} 
	\end{equation}

	\end{enumerate}
	\end{definition}
\begin{remark}
Observe that (\ref{eq1}) is indeed regular in the sense of \cite[definition 1.1]{BW1} since for a uniformly continuous Hessian, the coefficient given by (\ref{Bdef}) takes the form of the $b^{ij,kl}$ coefficient shown in \cite[(1.3)]{BW1} as $h\rightarrow 0$:
\begin{equation*}
b^{ij,kl}(D^2u(x))=a^{ij,kl}(D^2u(x))+\frac{\partial a^{pq,kl}}{\partial
	u_{ij}}(D^2u(x))u_{pq}(x) .
\end{equation*}

\end{remark}

\subsection{Preliminaries}
Next, we state the following preliminary results, which will be used to prove higher regularity.

\begin{theorem}
\label{five}\cite[Theorem 2.1]{BW1}. Suppose $w\in H^{2}(B_{r})$ satisfies the
uniformly elliptic constant coefficient equation%
\begin{align*}
\int c_{0}^{ik,jl}w_{ik}\eta_{jl}dx &  =0\label{ccoef}\\
\forall\eta &  \in C_{0}^{\infty}(B_{r}(0)).\nonumber
\end{align*}
%
Then for any $0<\rho\leq r$ there holds

\begin{align*}
\int_{B_{\rho}}|D^{2}w|^{2} &  \leq C_{1}(\rho/r)^{n}||D^{2}w||_{L^{2}(B_{r}%
)}^{2}\\
\int_{B_{\rho}}|D^{2}w-(D^{2}w)_{\rho}|^{2} &  \leq C_{2}(\rho/r)^{n+2}%
\int_{B_{r}}|D^{2}w-(D^{2}w)_{r}|^{2}%
\end{align*}
where $C_{1},C_{2}$ depend on the ellipticity constant and $(D^{2}w)_{\rho}$ is the
average value\ of $D^{2}w$ on a ball of radius $\rho$.
\end{theorem}

\medskip

\begin{corollary}
\label{Cor2} \cite[Corollary 2.2]{BW1}. \textit{Suppose }$w$\textit{ is as in Theorem \ref{five}. Then for any~}$u\in H^{2}(B_{r}),$ and\textit{ for
any~} $0<\rho\leq r,$ there holds%
\begin{equation*}
\int_{B_{\rho}}\left\vert D^{2}u\right\vert ^{2}\leq4C_{1}(\rho/r)^{n}%
\left\Vert D^{2}u\right\Vert _{L^{2}(B_{r})}^{2}+\left(  2+8C_{1}\right)
\left\Vert D^{2}(w-u)\right\Vert _{L^{2}(B_{r})}^{2}\label{twothree}%
\end{equation*}
	where $C_{1}$ depends on the ellipticity constant.

\end{corollary}

\medskip

\begin{lemma}
\label{HanLin} \cite[Lemma 3.4]{han}. Let $\phi$ be a nonnegative and
nondecreasing function on $[0,R].$ \ Suppose that
\[
\phi(\rho)\leq A\left[  \left(  \frac{\rho}{r}\right)  ^{\alpha}%
+\varepsilon\right]  \phi(r)+Br^{\beta}%
\]
for any $0<\rho\leq r\leq R,$ with $A,B,\alpha,\beta$ nonnegative constants
and $\beta<\alpha.$ \ Then for any $\gamma\in(\beta,\alpha),$ there exists a
constant $\varepsilon_{0}=\varepsilon_{0}(A,\alpha,\beta,\gamma)$ such that if
$\varepsilon<\varepsilon_{0}$ we have for all $0<\rho\leq r\leq R$%
\[
\phi(\rho)\leq c\left[  \left(  \frac{\rho}{r}\right)  ^{\gamma}%
\phi(r)+Br^{\beta}\right]
\]
where $c$ is a positive constant depending on $A,\alpha,\beta,\gamma.$ \ In
particular, we have for any $0<r\leq R$%
\[
\phi(r)\leq c\left[  \frac{\phi(R)}{R^{\gamma}}r^{\gamma}+Br^{\beta}\right]
.
\]

\end{lemma}

\subsection{Regularity results}
We first show that any $ W^{2,\infty}$ weak solution of the regular equation (\ref{eq1}) is in $W^{3,2}$. The proof follows from the arguments used in \cite[Lemma 3.1]{CW} and \cite[pg 4340-4341]{BW1}. We present it here for the sake of completeness. 
\begin{proposition}
\label{prop_2} Suppose that $u\in W^{2,\infty}(B_{1})$ is a weak solution of the regular equation (\ref{eq1}) on $B_{1}$ such that $D^2u(B_1)\subset U$. Then $u\in W^{3,2}(B_{1})$ 
and satisfies the following estimate
\begin{equation}
    ||u||_{W^{3,2}(B_{1/2})}\leq C(\Lambda,||u||_{W^{2,\infty}(B_{1})}). \label{est2}
\end{equation}

\end{proposition} 

\begin{proof}
 Let $\tau\in C_{c}^{\infty}\left(  B_{1}\right)$ be a cut off function in
$B_{1}$ that takes the value $1$ on $B_{1/2}$. Let $\eta=-[\tau^{4}u^{h_{m}}]^{-h_{m}}$ where the subscript $h_{m}$ denotes a difference quotient in the direction $e_m$. We chose $h$ small enough depending
on $\tau$ so that $\eta$ is well defined. By approximation (\ref{eq1}) holds for the above $\eta\in W_{0}^{2,\infty}(B_1)$:
\[
\int_{B_1}a^{ij,kl}(D^{2}u)u_{ij}\left[  \tau^{4}u^{h_{m}}\right]^{-h_m}
_{kl}dx=0.
\]
 Integrating by parts with respect to the difference quotient for $h$ small, we get%
\[
\int_{B_{1}}[a^{ij,kl}(D^{2}u)u_{ij}]^{h_{m}}[\tau^{4}u^{h_{m}}]_{kl}dx=0.
\]
We write the first difference quotient 
as
\begin{align}
\lbrack a^{ij,kl}(D^{2}u)u_{ij}]^{h_{m}}(x)  &  =a^{ij,kl}(D^{2}
u(x+h_m))\frac{u_{ij}(x+h_m)-u_{ij}(x)}{h}\nonumber\\
&  +\frac{1}{h}\left[  a^{ij,kl}(D^{2}u(x+h_m))-a^{ij,kl}(D^{2}%
u(x))\right]  u_{ij}(x)\nonumber\\
&  =a^{ij,kl}(D^{2}u(x+h_m))u^{h_m}_{ij}(x)\nonumber\\
&  +\left[  \int_{0}^{1}\frac{\partial a^{ij,kl}}{\partial u_{pq}}%
(tD^{2}u(x+h_m)+(1-t)D^{2}u(x))dt\right]  u^{h_m}_{pq}(x)u_{ij}(x).\nonumber
\end{align}
Denoting $v=u^{h_{m}}$ and using the notation in (\ref{Bdef}) we get
\begin{equation}
\int_{B_{1}}{b}^{ij,kl}v_{ij}[\tau^{4}v]_{kl}dx=0. \label{dq3}%
\end{equation}
Expanding derivatives we get
\begin{align}
\int_{B_{1}} b^{ij,kl}v_{ij}\tau^{4}v_{kl}dx=\int_{B_{1}}
 b^{ij,kl}v_{ik}\left(    (\tau^{4})
_{jl}v+  (\tau^{4})  _{l}v_{j}+ (\tau^{4})  _{j}%
v_{l}\right) dx.\nonumber
\end{align} 
By our assumption in (\ref{Bcondition}) $ b^{ij,kl}$ is uniformly elliptic on $U$. Therefore, we get
\begin{align*}
    \int_{B_{1}}\tau^{4}\Lambda|D^{2}v|^{2} dx\leq\int_{B_{1}}\left\vert
 b^{ij,kl}\right\vert \left\vert v_{ij}\right\vert \tau^{2}C%
(\tau,D\tau,D^{2}\tau)\left(  1+|v|+|Dv|\right)  dx\\
\leq C\sup_{B_1}  b
^{ij,kl}\int_{B_{1}}\left(  \varepsilon\tau^{4}|D^{2}v|^{2}+C\frac
{1}{\varepsilon}(1+|v|+|Dv|)^{2}\right)  dx.
\end{align*}
Choosing $\varepsilon>0$ appropriately, we get
\[
\int_{B_{1/2}}|D^{2}v|^{2}dx\leq C(||u||_{W^{2,\infty}(B_1)},\Lambda)\int_{B_{1}}(1+|v|+|Dv|)^{2}dx.
\]
Now this estimate is uniform in $h$ and direction
$e_{m,}$ so we conclude that the derivatives are in $W^{2,2}(B_{1/2}).$ This shows that%
\[
||D^{3}u||_{L^{2}(B_{1/2})}\leq C\left(  ||u||_{W^{2,\infty}(B_1)},\Lambda\right)  .
\]
\end{proof}
Observe that for the above result the Hessian was not required to have a small oscillation. Next, we prove $u$ is in $C^{2,\alpha}$ if the Hessian has a small oscillation. An explicit bound for the required modulus of oscillation is given in (\ref{omega}).

\begin{proposition} \label{hold}
\label{prop_1} Suppose that $u\in W^{2,\infty}(B_{1})$ is a weak solution of the regular equation (\ref{eq1}) on $B_{1}$ such that $D^2u(B_1)\subset U$. Let $\alpha \in (0,1)$. There exists $\omega(\Lambda,n, \alpha$, $||D^2u||_{L^{\infty}(B_1)})>0$ 
such that if $D^2u$ satisfies condition (\ref{BMO1}) with modulus $\omega$, then $D^2u\in C^{\alpha}(B_{1})$ and satisfies the following estimate
\begin{equation}
    ||D^{2}u||_{C^{\alpha}(B_{1/4})}\leq C(\Lambda,||u||_{W^{2,\infty}(B_{1})},\alpha). \label{est1}
\end{equation}
\end{proposition} 

\begin{proof}

We take a single difference quotient
\[
\int_{B_{1}}[a^{ij,kl}(D^{2}u)u_{ij}]^{h_{m}}\eta_{kl}dx=0
\]
and arrive at (\ref{dq3}) as above with $u^{h_{m}}=f$
\begin{equation}
    \int_{B_{1}}b^{ij,kl}f_{ij}\eta_{kl}dx=0. \label{lala}
\end{equation}

We pick an arbitrary point $x_0$ inside $B_{1/4}$ and consider this arbitrary point
to be the center of $B_{r}.$ We denote $B_r(x_0)$ by $B_r$ for the rest of this proof. 
For a fixed $r<3/4$ let $w$ solve the following boundary value
problem
\begin{align*}
\int_{B_{r}}(b^{ij,kl})_{1}w_{ij}\eta_{kl}dx  &  =0,\forall\eta\in C_{0}^{\infty
}(B_{r})\\
w  &  =f\text{ on }\partial B_{r}\\
D w  &  =D f\text{ on }\partial B_{r}.
\end{align*}
Note that the above PDE with the given boundary
condition has a unique solution and is smooth on the interior of
$B_{r}$ (\cite[Theorem 6.33]{Folland}). Denoting $v=f-w$ we see
\[
\int_{B_{r}}(b^{ij,kl})_1v_{ij}\eta_{kl}dx=\int_{B_{r}}\left(  (b^{ij,kl})_1
-b^{ij,kl}(x)\right)  f_{ij}\eta_{kl}dx.
\]
Since $v$ can be well approximated by smooth test
functions in $H^2_0(B_r)$ we use $v$ as a test
function. We get
\[
\int_{B_{r}}(b^{ij,kl})_1v_{ij}v_{kl}dx=\int_{B_{r}}\left(  (b^{ij,kl})_1-b^{ij,kl}(x)\right)  f_{ij}v_{kl}dx.
\]
We denote the small oscillation modulus of $D^2u$ by $\omega$ to be determined soon. Since $b^{ij,kl}$ is a smooth function of $D^2u$, it is  Lipschitz in $D^2u(B_{1})\subset U$. The small oscillation modulus of $b^{ij,kl}$ is lesser or equal to $C'\omega$ where $C'=||b^{ij,kl}||_{Lip(U)}$, which in turn is bounded by $C(||D^2u||_{L^{\infty}(B_{1})})$. Using this and uniform ellipticity (\ref{Bcondition}) we get
\begin{align}
\left(  \Lambda\int_{B_{r}}\left\vert D^{2}v\right\vert ^{2}dx\right)
^{2}
\leq C'\omega^2\int_{B_r}|D^2v|^2dx\int_{B_r}|D^2f|^2dx.\nonumber
\end{align}

So now we have
\[
\int_{B_{r}}\left\vert D^{2}v\right\vert ^{2}dx\leq\frac{C'\omega^2}
{\Lambda^{2}}\int_{B_{r}}\left\vert D^{2}f\right\vert ^{2}dx.
\]

Combining Corollary \ref{Cor2} with the above we get the following for any $0<\rho\leq r$:
\begin{align}
    \int_{B_\rho}|D^2f|^2dx\leq [\frac{(2+8C_1)C'\omega^2}{\Lambda^2}+4C_1(\rho/r)^n]\int_{B_r}|D^2f|^2dx.\label{2hl}
\end{align}

Now in order to apply Lemma \ref{HanLin} we choose
\begin{align*}
\phi(\rho) &  =\int_{B_{\rho}}\left\vert D^{2}f\right\vert ^{2}dx\\
A &  =4C_{1}\\
\varepsilon &  =\frac{\left(  2+8C_{1}\right)  }{\Lambda^{2}}C'\omega^2\\
\alpha & =n\\
\beta &=0 , B=0\\
\gamma &  =n-2+2\alpha\\
R &=1/4
\end{align*}
where the
notations appearing on the left hand side of the above table refer to
constants as they are named in Lemma \ref{HanLin}. We observe that (\ref{2hl}) can be written using notation on the
left side of the above table as
\begin{equation}
\phi(\rho)\leq A\left[  \left(  \frac{\rho}{r}\right)  ^{\alpha}%
+\varepsilon\right]  \phi(r)\label{this0}%
\end{equation}
for all $0<\rho\leq r<\frac{1}{4}.$ There exists a constant $\varepsilon
^{\ast}\left(  A,\alpha,\gamma\right) =\varepsilon
^{\ast}\left(  \Lambda,n,\alpha\right) $ so that (\ref{this0}) allows us to
conclude that there is a constant $C>0$ such that
\[
\phi(\rho)\leq C \left(  \frac{\rho}{r}\right)  ^{n-2+2\alpha}\phi
(r)
\]
whenever
\begin{equation*}
\frac{\left(  2+8C_{1}\right)  }{\Lambda^{2}}C'\omega^2\leq
\varepsilon^{\ast}\left(  \Lambda,n,\alpha\right)  .\label{eps_0_def}%
\end{equation*}
 We pick one such $\omega$: 
\begin{equation}
\omega^2\leq
\frac{\varepsilon^{\ast}\left(  \Lambda,n,\alpha\right)\Lambda^2}{(2+8C_1)C'} . \label{omega}
\end{equation}
Therefore we have 
\begin{align*}
\int_{B_r}|D^2f|^2 dx\leq Cr^{n-2+2\alpha}\int_{B_{1/2}}|D^2f|^2 dx
\end{align*}
where $C$ depends on $ \Lambda,n,\alpha$. 
Since we chose an arbitrary point in $B_{1/4}$, applying Morrey's
Lemma \cite[Lemma 3, page 8]{SimonETH} to $Df$ we get
\begin{align*}
|D(u^{h_m})|_{C^{\alpha}(B_{r})}\leq C(\Lambda,  ||u||_{W^{3,2}(B_{1/2})},\alpha ),
\end{align*}
 which combined with estimate (\ref{est2}) gives the desired estimate (\ref{est1}).
\end{proof}

Note that the above estimates are appropriately scaling invariant and therefore can be used to obtain interior estimates for a solution in the interior of any sized domain.

\begin{theorem}\label{1.1}
Suppose that $u\in W^{2,\infty}(B_{1})$ is a weak solution of the regular fourth order equation (\ref{eq1}) on $B_{1}$ such that $D^2u(B_1)\subset U$. There exists a $\omega(\Lambda,n,||D^2u||_{L^{\infty}(B_1)})>0$ such that 
	if $D^2u$ satisfies condition (\ref{BMO1}),
	then $u$ is smooth in $B_1$.
 \end{theorem}

\begin{proof}
From the above Propositions it follows that $u\in C^{2,\alpha}(B_1).$  Then smoothness follows from \cite[Theorem 1.2]{BW1}. 
\end{proof}

\begin{remark} \label{rem}
Observe that the result in Theorem \ref{1.1}  is not restricted to equations of the form (\ref{eq1}) but it also applies to equations of the following form with smooth coefficients in the Hessian
\begin{equation}
    \int_{B_1} G^{ij}(D^2u)\eta_{ij}dx=0 \label{100}
\end{equation}
as long as uniform ellipticity of its linearization (condition (\ref{Bcondition})) is maintained. In other words, we require 
\begin{equation}
\frac{\partial G^{ij}}{\partial u_{kl}}(\xi)\sigma_{ij}\sigma_{kl}\geq
\Lambda\left\Vert \sigma\right\Vert ^{2},\text{ $\forall$ }\sigma\text{ $\in
S^{n\times n}$} \label{Bb}%
\end{equation}
for any $\xi\in U $. One can check that the above observation is true by deriving a difference quotient
expression from (\ref{100}) in the direction $h_m$ to get\begin{equation}
\int_{B_{1}} \beta^{ij,kl}u^m_{ij}\eta_{kl} dx=0 \label{main3}%
\end{equation}
where 
\begin{align*}
    \beta^{ij,kl}=\int_{0}^{1}\frac{\partial G^{ij}}{\partial u_{kl}}(D^2u(x)+t[D^2u(x+h_m)-D^2u(x)])dt. 
\end{align*}This shows that the difference quotient of the solution of (\ref{100}) satisfies an equation of the form (\ref{lala}), which was previously derived from (\ref{eq1}). Since the equation that we work with is the one satisfied by the difference quotient, the result holds good for (\ref{100}).  
\end{remark}

\section{Regularity theory: second order}

In this section, we apply the methods used in the previous section to prove analogous regularity results for second order nonlinear equations in divergence form (\ref{eq2}).

We again start by introducing the following definition.
\begin{definition}[Regular equation]
	We define equation (\ref{eq2}) to be regular on $V$ when the 
	following conditions are satisfied on $V$:
	\begin{enumerate}
	    \item[(i)] The coefficients $a^{ij}$ depend smoothly on $Du$.
	    \item[(ii)] The linearization of (\ref{eq2}) is uniformly elliptic: the leading coefficient of the linearized equation given by
	    \begin{equation}
	        b^{ij}(Du(x))=\int_{0}^1\frac{\partial}{\partial{u_{i}}} \bigg[a^{kj}(Du(x)+t [Du(x+h_m)-Du(x)])u_{k}(x)\bigg]dt \label{Bdef2}
	    \end{equation}
	     satisfies the uniform ellipticity condition for any $\xi\in V$:
	    \begin{equation}
b^{ij}(\xi)\sigma_{i}\sigma_{k}\geq\Lambda\left\Vert \sigma\right\Vert \text{  }\forall \sigma \in \mathbb R^{n}.
\label{Bcondition2} 
	\end{equation}

	\end{enumerate}
	\end{definition}
\begin{remark}
Observe that for a uniformly continuous gradient, the coefficient given by (\ref{Bdef2}) takes the following form as $h\rightarrow 0$:
\begin{equation*}
b^{ij}(Du(x))=a^{ij}(Du(x))+\frac{\partial a^{kj}}{\partial
	u_{i}}(Du(x))u_{k}(x) .
\end{equation*}

\end{remark}

\subsection{Preliminaries}
Next, we state the following preliminary results, which will be used to prove higher regularity.
We have rephrased the following results from \cite{han} with notations used in this paper.
\begin{theorem}
\label{five2}\cite[Lemma 3.10]{han}. Suppose $w\in H^{1}(B_{r})$ satisfies the uniformly elliptic constant coefficient equation%
\begin{align*}
\int c_{0}^{ij}w_{i}\eta_{j}dx &  =0\label{ccoef}\\
\forall\eta &  \in C_{0}^{\infty}(B_{r}(0)).\nonumber
\end{align*}
%
Then for any $0<\rho\leq r$ there holds

\begin{align*}
\int_{B_{\rho}}|Dw|^{2} &  \leq c(\rho/r)^{n}||Dw||_{L^{2}(B_{r}%
)}^{2}\\
\int_{B_{\rho}}|Dw-(Dw)_{\rho}|^{2} &  \leq c(\rho/r)^{n+2}%
\int_{B_{r}}|Dw-(Dw)_{r}|^{2}%
\end{align*}
where $c$ depends on the ellipticity constants and $(Dw)_{\rho}$ is the
average value\ of $Dw$ on a ball of radius $\rho$.
\end{theorem}

\medskip

\begin{corollary}
\label{Cor22} \cite[Corollary 3.11]{han}. \textit{Suppose }$w$\textit{ is as in
the Theorem \ref{five2}. Then for any~}$u\in H^{1}(B_{r}),$ and\textit{ for
any~} $0<\rho\leq r,$ there holds%
\begin{equation}
\int_{B_{\rho}}\left\vert Du\right\vert ^{2}\leq c[(\rho/r)^{n}%
\left\Vert Du\right\Vert _{L^{2}(B_{r})}^{2}+
\left\Vert D(w-u)\right\Vert _{L^{2}(B_{r})}^{2}]\label{twothree}%
\end{equation}
	where $c$ depends on the ellipticity constants.

\end{corollary}

\medskip

\subsection{Regularity results}
We first show that any $ W^{1,\infty}$ weak solution of the regular equation (\ref{eq2}) is in $W^{2,2}$.  
\begin{proposition}
\label{prop_22} Suppose that $u\in W^{1,\infty}(B_{1})$ is a weak solution of the regular equation (\ref{eq2}) on $B_{1}$ such that $Du(B_1)\subset V$. Then $u\in W^{2,2}(B_{1})$ 
and satisfies the following estimate
\begin{equation}
    ||u||_{W^{2,2}(B_{1/2})}\leq C(\Lambda,||u||_{W^{1,\infty}(B_{1})}). \label{est22}
\end{equation}

\end{proposition} 

\begin{proof}
 Let $\tau\in C_{c}^{\infty}\left(  B_{1}\right)$ be a cut off function in
$B_{1}$ that takes the value $1$ on $B_{1/2}$. Let $\eta=-[\tau^{2}u^{h_{m}}]^{-h_{m}}$ where again the subscript $h_{m}$ denotes a difference quotient in the direction $e_m$. We chose $h$ small enough depending
on $\tau$ so that $\eta$ is well defined. By approximation (\ref{eq2}) holds for the above $\eta\in W_{0}^{1,\infty}(B_1)$:
\[
\int_{B_1}a^{ij}(Du)u_{i}\left[  \tau^{2}u^{h_{m}}\right]^{-h_m}
_{j}dx=0.
\]
Integrating by parts with respect to the difference quotient for $h$ small, we get%
\[
\int_{B_{1}}[a^{ij}(Du)u_{i}]^{h_{m}}[\tau^{2}u^{h_{m}}]_{j}dx=0.
\]
We write the first difference quotient 
as
\begin{align}
\lbrack a^{ij}(Du)u_{i}]^{h_{m}}(x)  &  =a^{ij}(D
u(x+h_m))\frac{u_{i}(x+h_m)-u_{i}(x)}{h}\nonumber\\
&  +\frac{1}{h}\left[  a^{ij}(Du(x+h_m))-a^{ij}(D
u(x))\right]  u_{i}(x)\nonumber\\
&  =a^{ij}(Du(x+h_m))u^{h_m}_{i}(x)\nonumber\\
&  +\left[  \int_{0}^{1}\frac{\partial a^{ij}}{\partial u_{p}}%
(tDu(x+h_m)+(1-t)Du(x))dt\right]  u^{h_m}_{p}(x)u_{i}(x).\nonumber
\end{align}
Denoting $v=u^{h_{m}}$ and using the notation in (\ref{Bdef2}) we get
\begin{equation}
\int_{B_{1}}{b}^{ij}v_{i}[\tau^{2}v]_{j}dx=0. \label{dq32}%
\end{equation}
Expanding derivatives we get
\begin{align}
\int_{B_{1}} b^{ij}v_{i}\tau^{2}v_{j}dx=\int_{B_{1}}
 b^{ij}v_{i}  (\tau^{2})
_{j}v
 dx.\nonumber
\end{align} 
By our assumption in (\ref{Bcondition2}) $ b^{ij}$ is uniformly elliptic on $V$. Therefore, we get
\begin{align*}
\int_{B_{1}}\tau^{2}\Lambda|Dv|^{2} dx\leq \int_{B_{1}}\left\vert
 b^{ij}\right\vert \left\vert v_{i}\right\vert \tau C
(\tau,D\tau)|v|  dx\\
\leq C\sup_{B_1}  b
^{ij}\int_{B_{1}}\left(  \varepsilon\tau^{2}|Dv|^{2}+C\frac
{1}{\varepsilon}|v|^{2}\right)  dx.
\end{align*}
Choosing $\varepsilon>0$ appropriately, we get
\[
\int_{B_{1/2}}|Dv|^{2}dx\leq C(||u||_{W^{1,\infty}(B_1)},\Lambda)\int_{B_{1}}|v|^{2}dx.
\]
Now this estimate is uniform in $h$ and direction
$e_{m,}$ so we conclude that the derivatives are in $W^{1,2}(B_{1/2}).$ This shows that%
\[
||D^{2}u||_{L^{2}(B_{1/2})}\leq C\left(  ||u||_{W^{1,\infty}(B_1)},\Lambda\right)  .
\]
\end{proof}
Observe that for the above result the gradient was not required to have a small oscillation. Next, we prove $u$ is in $C^{1,\alpha}$ if the gradient has a small oscillation and an explicit bound for the required modulus of oscillation is given in (\ref{omega2}).
\begin{proposition} \label{hold}
\label{prop_1} Suppose that $u\in W^{1,\infty}(B_{1})$ is a weak solution of the regular equation (\ref{eq2}) on $B_{1}$ such that $Du(B_1)\subset V$. Let $\alpha \in (0,1)$. There exists $\omega(\Lambda,n, \alpha,||Du||_{L^{\infty}(B_1)})>0$ 
such that if $Du$ satisfies condition (\ref{BMO1}) with modulus $\omega$, then $Du\in C^{\alpha}(B_{1})$ and satisfies the following estimate
\begin{equation}
    ||Du||_{C^{\alpha}(B_{1/4})}\leq C(\Lambda,||u||_{W^{1,\infty}(B_{1/2})},\alpha). \label{est12}
\end{equation}
\end{proposition} 

\begin{proof}

We take a single difference quotient
\[
\int_{B_{1}}[a^{ij}(Du)u_{i}]^{h_{m}}\eta_{j}dx=0
\]
and arrive at (\ref{dq32}) as before with $u^{h_{m}}=f$
\begin{equation}
    \int_{B_{1}}b^{ij}f_{i}\eta_{j}dx=0. \label{lala2}
\end{equation}

We pick an arbitrary point $x_0$ inside $B_{1/4}$ and consider this arbitrary point
to be the center of $B_{r}(x_0).$
For a fixed $r<3/4$ we let $w$ solve the following boundary value
problem
\begin{align*}
\int_{B_{r}}(b^{ij})_{1}w_{i}\eta_jdx  &  =0,\forall\eta\in C_{0}^{\infty
}(B_{r})\\
w  &  =f\text{ on }\partial B_{r}.
\end{align*}
Note that the above PDE with the given boundary
condition has a unique solution and is smooth on the interior of
$B_{r}$ (\cite{Folland}). Let $v=f-w.$ Note that%
\[
\int_{B_{r}}(b^{ij})_1v_{i}\eta_{j}dx=\int_{B_{r}}[ (b^{ij})_1
-b^{ij}(x)] f_{i}\eta_{j}dx.
\]
Since $v$ can be well approximated by smooth test
functions in $H^1_0(B_r)$ we can use $v$ as a test
function. We get
\[
\int_{B_{r}}(b^{ij})_1v_{i}v_{j}dx=\int_{B_{r}}[ (b^{ij})_1-b^{ij}(x)]  f_{i}v_{j}dx.
\]
We denote the small oscillation modulus of $Du$ by $\omega$ to be determined soon. Since $b^{ij}$ is a smooth function of $Du$, it is  Lipschitz in $Du(B_{1})\subset V$. The small oscillation modulus of $b^{ij}$ is lesser or equal to $C''\omega$ where $C''=||b^{ij}||_{Lip(V)}$, which in turn is bounded by $C(||Du||_{L^{\infty}(B_{1})})$. Using this and uniform ellipticity (\ref{Bcondition2}) we get
\begin{align}
\left(  \Lambda\int_{B_{r}}\left\vert Dv\right\vert ^{2}dx\right)
^{2}
\leq C''\omega^2\int_{B_r}|Dv|^2dx\int_{B_r}|Df|^2dx.\nonumber
\end{align}

So now we have
\[
\int_{B_{r}}\left\vert Dv\right\vert ^{2}dx\leq\frac{C''\omega^2}
{\Lambda^{2}}\int_{B_{r}}\left\vert Df\right\vert ^{2}dx.
\]

Combining Corollary \ref{Cor22} with the above we get the following for any $0<\rho\leq r$:
\begin{align}
    \int_{B_\rho}|Df|^2dx\leq c[\frac{C''\omega^2}{\Lambda^2}+(\rho/r)^n]\int_{B_r}|Df|^2dx.\label{2hl2}
\end{align}

Now in order to apply Lemma \ref{HanLin} we choose
\begin{align*}
\phi(\rho) &  =\int_{B_{\rho}}\left\vert Df\right\vert ^{2}dx\\
A &  =c\\
\varepsilon &  =\frac{  c C''\omega^2 }{\Lambda^{2}}\\
\alpha & =n\\
\beta &=0, B=0\\
\gamma &  =n-2+2\alpha\\
R &=1/4
\end{align*}
where the
notations appearing on the left hand side of the above table refer to
constants as they are named in Lemma \ref{HanLin}. We observe that (\ref{2hl2}) can be written using notation on the
left side of the above table as
\begin{equation}
\phi(\rho)\leq A\left[  \left(  \frac{\rho}{r}\right)  ^{\alpha}%
+\varepsilon\right]  \phi(r)\label{this}%
\end{equation}
for all $0<\rho\leq r<\frac{1}{4}.$ There exists a constant $\varepsilon
^{\ast}\left(  A,\alpha,\gamma\right) =\varepsilon
^{\ast}\left(  \Lambda,n,\alpha\right) $ so that (\ref{this}) allows us to
conclude that there is a constant $C>0$ such that
\[
\phi(\rho)\leq C \left(  \frac{\rho}{r}\right)  ^{n-2+2\alpha}\phi
(r)
\]
whenever
\begin{equation*}
\frac{  c C''\omega^2 }{\Lambda^{2}}\leq
\varepsilon^{\ast}\left(  \Lambda,n,\alpha\right)  .\label{eps_0_def}%
\end{equation*}
 We pick one such $\omega$: 
\begin{equation}
\omega^2\leq
\frac{\varepsilon^{\ast}\left(  \Lambda,n,\alpha\right)\Lambda^2}{cC''} . \label{omega2}
\end{equation}
Therefore we have 
\begin{align*}
\int_{B_r}|Df|^2 dx\leq Cr^{n-2+2\alpha}\int_{B_{1/2}}|Df|^2 dx.
\end{align*}
where $C$ depends on $ \Lambda,n,\alpha$. 
Since we chose an arbitrary point in $B_{1/4}$, applying Morrey's
Lemma \cite[Lemma 3, page 8]{SimonETH} to $f$ we get
\begin{align*}
|u^{h_m}|_{C^{\alpha}(B_{r})}\leq C(\Lambda,  ||u||_{W^{2,2}(B_{1/2})},\alpha ),
\end{align*}
 which combined with estimate (\ref{est22}) gives the desired estimate (\ref{est12}).
\end{proof}

\begin{theorem}\label{1.2}
Suppose that $u\in W^{1,\infty}(B_{1})$ is a weak solution of the regular second order equation (\ref{eq2}) on $B_{1}$ such that $Du(B_1)\subset V$. There exists a $\omega(\Lambda,n,||Du||_{L^{\infty}(B_1)})>0$ such that 
	if $Du$ satisfies condition (\ref{BMO1}),
	then $u$ is smooth in $B_1$.
 \end{theorem}

\begin{proof}
From the above Propositions it follows that $u\in C^{1,\alpha}(B_1).$ This leads to H\"older continuous coefficients of (\ref{lala2}). Applying \cite[Theorem 3.13]{han} to the divergence form equation (\ref{lala2}), we get $u\in C^{2,\alpha}(B_1).$ Following a standard bootstrapping procedure, we get $u$ is smooth in $B_1$.
\end{proof}

\begin{remark} \label{rem2}
Again similar to remark (\ref{rem}) observe that the result in Theorem \ref{1.2}  is not restricted to equations of the form (\ref{eq2}) but it also applies to equations of the following form with smooth coefficients in the gradient
\begin{equation*}
    \int_{B_1} G^{i}(Du)\eta_{i}dx=0 \label{1}
\end{equation*}
as long as uniform ellipticity of its linearization (condition (\ref{Bcondition2})) is maintained:
\begin{equation}
\frac{\partial G^{i}}{\partial u_{k}}(\xi)\sigma_{i}\sigma_{k}\geq
\Lambda\left\Vert \sigma\right\Vert ^{2},\text{ $\forall$ }\sigma\text{ $\in\mathbb R^{n}$} \label{Bb2}%
\end{equation}
where $\xi\in V $. 
\end{remark}

As a direct application of Theorem \ref{1.2}, we see that for a weak $C^{0,1}(B_{1})$ solution $u$ of the minimal surface equation (\ref{ms}) there exists $\omega(n,||Du||_{L^{\infty}(B_1)})>0$ such that if the gradient has a small oscillation bounded by $\omega$,  then $u$ is smooth in $B_1$. 

\section{Proofs of the main results}

\noindent
\textit{Proof of Theorem \ref{main1}.} 
Since $u$ is a critical point of (\ref{Ffunc}) and we are restricted to compactly supported variations, from the Euler Lagrange form we see that $u$ is a weak solution of the following equation:
\begin{equation*}
    \int_{B_1}\frac{\partial F }{\partial{u_{ij}}}(D^2u)\eta_{ij}dx=0. \label{EL}
\end{equation*}
Since $F$ is uniformly convex we observe that
\[\frac{\partial}{\partial u_{kl}}\left(
	\frac{\partial F(D^{2}u)}{\partial u_{ij}}\right)  \sigma_{ij}\sigma_{kl}\geq
	\Lambda\left\vert \sigma\right\vert ^{2}.\]
Therefore $\frac{\partial F}{\partial u_{ij}}(D^2u)$ satisfies condition (\ref{Bb}). Hence in light of remark \ref{rem}, by Theorem \ref{1.1}  we have $u$ is smooth inside $B_1$.
\hfill$\square$
\bigskip
    
\noindent
\textit{Proof of Theorem \ref{main2}.}
Since $u$ is a critical point of (\ref{2nd}) and we are restricted to compactly supported variations, from the Euler Lagrange form we see that $u$ is a weak solution of the following equation:
\begin{equation*}
    \int_{B_1}\frac{\partial F }{\partial{u_{i}}}(Du)\eta_{i}dx=0. \label{EL}
\end{equation*}
Since $F$ is uniformly convex we observe that
\[\frac{\partial}{\partial u_{j}}\left(
	\frac{\partial F(Du)}{\partial u_{i}}\right)  \sigma_{i}\sigma_{j}\geq
	\Lambda\left\vert \sigma\right\vert ^{2}.\]
Therefore $\frac{\partial F}{\partial u_{i}}(Du)$ satisfies condition (\ref{Bb2}). Hence in light of remark (\ref{rem2}), by Theorem \ref{1.2}  we have $u$ is smooth inside $B_1$.
\hfill$\square$

\begin{remark}
As a quick application of Theorem \ref{main1}, we see that for a $C^{1,1}(B_{1})$ weak solution $u$ of the Hamiltonian stationary equation in double divergence form (\ref{hstat}),
there exists $\omega(n,||D^2u||_{L^{\infty}(B_1)})>0$ such that if the Hessian has a small oscillation bounded by $\omega$, then $u$ is smooth in $B_1$. In \cite[Theorem 1.1]{CW}, it was shown that a $C^{1,1}$ weak solution of (\ref{hstat}) is smooth if the $C^{1,1}$ norm of the solution is bounded by a small dimensional constant. The arguments used in \cite{CW} relied on proving an equivalence of (\ref{hstat}) with its geometric form where the fourth order operator factors into two nonlinear second order operators. However, it is important to note that such a factorization for any general fourth order double divergence equation is not always possible.

This naturally leads to the open question of whether the small oscillation requirement for the Hessian can be dropped: or in the case of the Hamiltonian stationary equation, can small oscillation of the Hessian be achieved by exploiting the geometric properties of the equation. In two dimensions, a $C^{1,1}$ solution alone is sufficient to prove smoothness, as shown in \cite{BW2}.

\end{remark}

\begin{remark}
Theorem \ref{1.1}  holds good if $b^{ij,kl}$ is replaced by $-b^{ij,kl}$ in condition (\ref{Bcondition}) and therefore our main result in Theorem \ref{main1} holds good if $F$ is uniformly concave. Similarly the second order results hold good if $b^{ij}$ is replaced by $-b^{ij}$ in condition (\ref{Bcondition2}) and $F$ is uniformly concave in Theorem \ref{main2}.
\end{remark}


\textbf{Acknowledgments.} The author is grateful to Y. Yuan for insightful discussions. The author thanks R. Shankar and M. Warren for helpful comments.

\bibliographystyle{amsalpha}
\bibliography{bootreg}

\end{document}